\documentclass{amsart}

\usepackage{amsmath, amssymb,amsthm}
\usepackage{MnSymbol}
\usepackage{color}
\usepackage{graphicx}
\usepackage{colortbl}
\usepackage{mathtools}
\usepackage{float}

\newtheorem{theorem}{Theorem}[section]
\newtheorem{definition}[theorem]{Definition}
\newtheorem{lemma}[theorem]{Lemma}

\newtheorem{proposition}[theorem]{Proposition}

\newlength\cellsize 

\savebox2{%
\begin{picture}(13,13)
\put(0,0){\line(1,0){13}}
\put(0,0){\line(0,1){13}}
\put(13,0){\line(0,1){13}}
\put(0,13){\line(1,0){13}}
\end{picture}}

\savebox3{%
\begin{picture}(13,13)
\put(0,0){\line(1,0){17}}
\put(0,13){\line(1,0){17}}
\end{picture}}

\savebox4{%
\begin{picture}(13,13)
\put(0,0){\line(1,0){17}}
\put(0,13){\line(1,0){17}}
\end{picture}}

\newcommand\cellify[1]{\def\thearg{#1}\def\nothing{}%
\ifx\thearg\nothing\vrule width0pt height\cellsize depth0pt%
  \else\hbox to 0pt{\usebox2\hss}\fi%
  \vbox to 13\unitlength{\vss\hbox to 13\unitlength{\hss$#1$\hss}\vss}}

\newcommand\tableau[1]{\vtop{\let\\=\cr
\setlength\baselineskip{-12000pt}
\setlength\lineskiplimit{12000pt}
\setlength\lineskip{0pt}
\halign{&\cellify{##}\cr#1\crcr}}}

\newcommand{\graybox}{\textcolor[RGB]{165,165,165}{\rule{1\cellsize}{1\cellsize}}\hspace{-\cellsize}\usebox2}

\newcommand{\grayboxtwo}{\textcolor[RGB]{220,220,220}{\rule{1\cellsize}{1\cellsize}}\hspace{-\cellsize}\usebox2}

\newcommand{\whitebox}{\textcolor[RGB]{255,255,255}{\rule{1\cellsize}{1\cellsize}}\hspace{-1.2\cellsize}\usebox3}

\title[quasi-Yamanouchi tableaux]{Enumerating quasi-Yamanouchi tableaux\\ of Durfee size two}  

\author[George Wang]{George Wang}
\address{Department of Mathematics, University of Pennsylvania, Philadelphia, PA, 19104}
\email{wage@math.upenn.edu}

\begin{document}


\date{\today}



\begin{abstract}
Quasi-Yamanouchi tableaux connect the two most studied types of tableaux. They are a subset of semistandard Young tableaux that are also a refinement on standard Young tableaux, and they can be used to improve the fundamental quasisymmetric expansion of Schur polynomials. We prove a product formula for enumerating certain quasi-Yamanouchi tableaux and provide strong evidence that no product formula exists in general for other shapes. Along the way, we also prove some nice properties of their distribution and symmetry. 
\end{abstract}

\maketitle

\section{Introduction}

Schur polynomials can be expressed elegantly as a sum of monomials indexed by semistandard Young tableaux. In 1984, Gessel \cite{Ge84} defined the fundamental basis for quasisymmetric functions and proved that Schur polynomials can be expressed using fundamental quasisymmetric polynomials indexed by standard Young tableaux instead. An advantage to this is that the number of standard Young tableaux of a given shape is fixed, while the number of semistandard Young tableaux depends on which integers are allowed. In some cases however, this expansion includes standard Young tableaux which contribute nothing to the polynomial, and a question arises as to whether this can be improved further. Assaf and Searles \cite{AS16} introduced the concept of quasi-Yamanouchi tableaux and proved that Gessel's expansion can be tightened by summing over these tableaux instead, which give exactly the nonzero terms. 

A natural question that arises when dealing with any kind of tableaux is how to count them. For example, Frame, Robinson, and Thrall produced the celebrated hook-length formula that counts standard Young tableaux of a given shape 
[FRT54], and Stanley enumerated semistandard Young tableaux with the hook-content formula [Sta71]. Both of these important results are product formulas over the cells of the diagram. Counting formulas and their various proofs often offer deep insights to the objects in question, and even the nonexistence of certain types of formulas is valuable knowledge.
In this paper, we prove a product formula for the enumeration of quasi-Yamanouchi tableaux with Durfee size two and show that it is unlikely that a product formula exists for tableaux with Durfee size greater than two due to the presence of large primes. We also prove several interesting properties of quasi-Yamanouchi tableaux along the way, demonstrating the symmetry in their distribution.

\section{Preliminaries}

A \textit{partition} $\lambda = (\lambda_1 \geq \lambda_2 \geq \cdots \geq \lambda_k > 0)$ is a decreasing sequence of positive integers. The \textit{size} of $\lambda$ is the sum of the integers of the sequence. The length of $\lambda$, denoted $\ell(\lambda)$, is the number of integers in the partition. We will use French notation to draw the \textit{diagram} corresponding to a partition by having the number of boxes in the $i$th row equal to $\lambda_i$ and left justifying rows. A box can be identified as $u = (i,j)$, where $i$ is the column and $j$ is the row. The \textit{content} of a square is $c(u) = i-j$. The hook-length of a box is $h(u)$, the number of boxes $(a,b)$ such that either $a > i$ or $b > j$, or $a=i, b=j$. 
\begin{figure}[t]
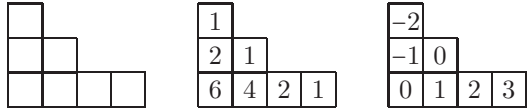

	\begin{displaymath}
	\tableau{
	\ \\
	\ & \ \\
	\ & \ & \ & \ \\
	}
	\ \ \ \ \ \ 
	\tableau{
	1\\
	2&1\\
	6&4&2&1\\
	}
	\ \ \ \ \ \ 
	\tableau{
	-2 \\ 
	-1 & 0 \\
	0 & 1 & 2 & 3\\
	}
	\end{displaymath}
	\caption{\label{fig:421} From left to right: the diagram, hook-lengths, and content of $(4, 2, 1)$.}
\end{figure}
The \textit{Durfee size} of a partition is the side length of the largest square that will fit in the diagram.

\begin{figure}[ht]
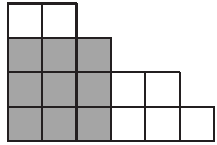

	\begin{displaymath}
	\tableau{
	\ & \ \\
	\graybox & \graybox &\graybox \\
	\graybox & \graybox &\graybox  & \ & \ \\
	\graybox & \graybox &\graybox  & \ & \ & \ \\
	}
	\end{displaymath}
	\caption{\label{fig:durfee} The Durfee square of $(6, 5, 3, 2)$, which has Durfee size $3$.}
\end{figure}

A \textit{semistandard Young tableau} is a filling of a partition $\lambda$ using positive integers that weakly increase to the right and strictly increase upwards. If there are no restrictions on the positive integers that can be used, then there is an infinite number of possible fillings, so it is more useful to enumerate SSYT$_m(\lambda)$, the number of semistandard fillings of $\lambda$ where the maximum allowed integer is $m$. Stanley's hook-content formula gives this enumeration.
\begin{theorem}[{Hook-content formula [Sta71]}]
Given a partition $\lambda$, the number of semistandard fillings of $\lambda$ where the maximum allowed integer is $m$ is

$$\textup{SSYT}_n(\lambda) = \prod_{u\in \lambda} \frac{m+c(u)}{h(u)}$$
\end{theorem}

\begin{figure}[ht]
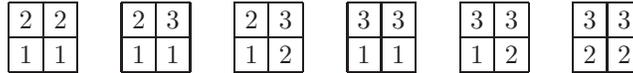

	\begin{displaymath}
	\tableau{
	2 & 2 \\
	1 & 1 \\ }
	\ \ \ \ \ \tableau{
	2 & 3 \\ 
	1 & 1 \\}
	\ \ \ \ \ \tableau{
	2 & 3 \\
	1 & 2 \\}
	\ \ \ \ \ \tableau{
	3 & 3 \\
	1 & 1 \\}
	\ \ \ \ \ \tableau{
	3 & 3 \\
	1 & 2 \\}
	\ \ \ \ \ \tableau{
	3 & 3 \\
	2 & 2\\}
	\end{displaymath}
	\caption{\label{fig:ssyt22} The $6$ semistandard fillings of SSYT$_3(2,2)$.}
\end{figure}

Let $|\lambda| = n$. A \textit{standard Young tableau} is a semistandard Young tableau using each of $\{1, 2, \ldots, n\}$ exactly once. Frame, Robinson, and Thrall provided an enumeration for standard Young tableaux using the hook-length formula.

\begin{theorem}[{Hook-length formula [FRT54]}]
Given a partition $\lambda$ of size $n$, the number of standard fillings of $\lambda$ is

$$\textup{SYT}(\lambda) = \frac{n!}{\prod_{u\in \lambda} h(u)}$$

\end{theorem}

\begin{figure}[ht]
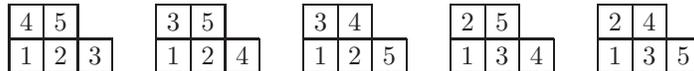

	\begin{displaymath}
	\tableau{
	4&5\\
	1&2&3\\}
	\ \ \ \ \ \tableau{
	3&5\\
	1&2&4}
	\ \ \ \ \ \tableau{
	3&4\\
	1&2&5}
	\ \ \ \ \ \tableau{
	2&5\\
	1&3&4}
	\ \ \ \ \ \tableau{
	2&4\\
	1&3&5}
	\end{displaymath}
	\caption{\label{fig:syt221} The $5$ standard fillings of $(3,2)$.}
\end{figure}

The \textit{descent set} of a tableau is the set of entries $i \in$ Des$(T) \subseteq \{1, 2, \ldots, n-1\}$ such that $i+1$ is weakly left of $i$. If we write the descent set as $\{d_1, d_2, \ldots, d_{k-1}\}$ in increasing order, then the first $\textit{run}$ of the tableau is the set of boxes that contain all the entries from $1$ to $d_1$. For $ 1 < i < k$, the $i$th run is the set of boxes that contain all entries from $d_{i-1}+1$ to $d_i$. The $k$th run starts at $d_k+1$ and ends at $n$. 

\begin{figure}[H]
	\begin{displaymath}
	\tableau{
	\mathclap{\graybox}\mathclap{\raisebox{4\unitlength}{9}}&\mathclap{\graybox}\mathclap{\raisebox{4\unitlength}{10}}& 12\\
	4&5&7&\mathclap{\graybox}\mathclap{\raisebox{4\unitlength}{11}}\\
	\mathclap{\grayboxtwo}\mathclap{\raisebox{4\unitlength}{1}}&\mathclap{\grayboxtwo}\mathclap{\raisebox{4\unitlength}{2}}&\mathclap{\grayboxtwo}\mathclap{\raisebox{4\unitlength}{3}} & 6&8\\
	}
	\end{displaymath}
	\caption{\label{fig:runs} A standard Young tableau with five runs. The first and fourth runs are highlighted, and the tableau has descent set $\{3,6,8,11\}$.}
\end{figure}

 Finally, we can define what it means for a tableau to be quasi-Yamanouchi. 

\begin{definition}
\normalfont A semistandard Young tableau is \textit{quasi-Yamanouchi} if when $i$ appears in the tableau, the leftmost instance of $i$ is weakly left of some $i-1$. Let QYT$_{\leq m}(\lambda)$ denote the set of quasi-Yamanouchi fillings of $\lambda$ where the entries of the tableau are at most $m$. Let QYT$_{=m}(\lambda)$ denote the set of quasi-Yamanouchi fillings of $\lambda$ where $m$ is exactly the largest value that appears.
\end{definition}
\begin{figure}[h]
	\begin{displaymath}
	\tableau{
	4\\
	2&3\\
	1&2&2&4\\}
	\ \ \ \ \ \ 
	\tableau{
	4\\
	3&3\\
	1&2&2&5\\}
	\end{displaymath}
	\caption{\label{fig:notqyt421} On the left is an example of a quasi-Yamanouchi filling and on the right is a non-example.}
\end{figure}
\begin{figure}[h]
	\begin{displaymath}
	\tableau{
	3\\
	2&2\\
	1&1\\}
	\ \ \ \ \ \tableau{
	3\\
	2&3\\
	1&1\\}
	\ \ \ \ \ \tableau{
	3\\
	2&3\\
	1&2\\}
	\ \ \ \ \ \tableau{
	4\\
	2&3\\
	1&2\\}
	\ \ \ \ \ \tableau{
	3\\
	2&4\\
	1&3\\}
	\end{displaymath}
	\caption{\label{fig:qyt221} All the quasi-Yamanouchi tableaux of shape $(2,2,1)$, demonstrating that $Q_{=3}(2,2,1)=3$ and $Q_{=4}(2,2,1) = 2$.}
\end{figure}
We will need the \textit{destandardization} map that Assaf and Searles \cite{AS16} use (Definition 2.5).
\begin{definition}[{[AS16]}]
\normalfont The \textit{destandardization} of a tableau $T$, denoted by dst$(T)$, is the tableau constructed as follows. If the leftmost $i$ lies strictly right of the rightmost $i-1$ or there are no $i-1$ entries, then decrement every $i$ to $i-1$. Repeat until no more $i$ entries can be decremented.
\end{definition}
\begin{figure}[h]
	\begin{displaymath}
	\tableau{
	6&9\\
	3&4&7\\
	1&2&5&8\\}
	\ \ \raisebox{-10\unitlength}{$\longrightarrow$} \ \ 
	\tableau{
	3&4\\
	2&2&3\\
	1&1&2&3\\}
	\end{displaymath}
	\caption{\label{fig:destandardization} The destandardization of a semistandard tableau which also happens to be standard.}
\end{figure}
We also reproduce Lemma $2.6$ from their paper below, where we will use part $(3)$ and $(4)$ later.
\begin{lemma}[{[AS16]}] \label{SSYT} The destandardization map is well-defined and satisfies the following

\text{\normalfont(1)} for $T\in \text{\normalfont SSYT}_n(\lambda), \text{\normalfont dst}(T) \in \text{\normalfont QYT}_{\leq n}(\lambda)$;

\text{\normalfont(2)} for $T\in \text{\normalfont SSYT}_n(\lambda), \text{\normalfont dst}(T) = T$ if and only if $T\in \text{\normalfont QYT}_{\leq n}(\lambda)$;

\text{\normalfont(3)} $\text{\normalfont dst : SSYT}_n(\lambda) \rightarrow \text{\normalfont QYT}_{\leq n}(\lambda)$ is surjective; and 

\text{\normalfont(4)} $\text{\normalfont dst : SSYT}_n(\lambda) \rightarrow \text{\normalfont QYT}_{\leq n}(\lambda)$ is injective if and only if $n\leq \ell(\lambda)$.
\end{lemma}

\section{Basic properties of quasi-Yamanouchi tableaux}

It is useful for the later results of this paper to know the range of maximum values that can appear in a quasi-Yamanouchi tableau. An easy lower bound is the number of rows, which follows from the definition. An easy upper bound is the number of boxes, since we cannot skip integers. It turns out that we can do a little better than this though.

\begin{lemma}\label{range}
Let $\lambda$ be a partition of $n$. If $T\in \text{\normalfont QYT}_{\leq n}(\lambda)$, then the maximum integer value that appears must be between $\ell(\lambda)$ and $n-(\lambda_1-1)$.
\end{lemma}

\begin{proof}
 We cannot have a maximum value less than $k$, otherwise our filling is not even semistandard. The lower bound is sharp however, as we can have a quasi-Yamanouchi filling where every entry in row $i$ of $T$ is $i$. The upper bound is sharp as well, and we will first show that it is achievable with a general construction. Start at the bottom of the first column and fill numbers in from the bottom to top, starting with $1$ and incrementing by $1$ at each step. Once the top of a column is reached, take the value at the top of this column and enter it again at the bottom of the next, then move upwards and continue to increment by $1$ until the top of the next column and repeat the process.

\begin{figure}[h]
	\begin{displaymath}
	\tableau{
	3&5&7\\
	2&4&6&8\\
	1&3&5&7\\}
	\end{displaymath}
	\caption{\label{fig:genconst1} The general construction applied to $(4,4,3)$.}
\end{figure}

We now prove inductively that this is the maximum achievable value. Assume that the bound holds for all shapes of size $n-1$. If $\lambda$ has more than one row, remove any corner box from $\lambda$ that is not in the first row, leaving a shape with $n-1$ boxes. By assumption, the highest achievable number in the resulting shape is known, and let this number be $m$. Since the box removed was not in the first row, $\lambda_1$ stays the same and $m = (n-1)-(\lambda_1-1)$. Now we add the box back. We cannot put any number greater than $m+1$ in the box; this would imply $m+1$ must appear elsewhere, which contradicts the inductive assumption. Thus the highest achievable value must be less than or equal to $m+1 = n-(\lambda_1-1)$.

If $\lambda$ has only one row, then the only quasi-Yamanouchi filling is the one where every entry is $1$, but in this case $n - (\lambda_1-1) = 1$ is clearly both an upper and lower bound.
\end{proof}

A related statement is also true, which is that for every value of $m$ within this range, QYT$_{=m}(\lambda) \neq 0$. This can also be shown by a constructive example. Start with a diagram and fill numbers in as if following the previous construction demonstrating the upper bound. Continue until reaching the first column where $m$ would appear. Fill this column as in the previous construction, but then decrement every element of the column until $m$ is at the top. For every subsequent column, fill every box with the entry directly to the left.

\begin{figure}[h]
	\begin{displaymath}
	\tableau{
	4&7\\
	3&6&8&8\\
	2&5&7&7&7\\
	1&4&6&6&6\\}
	\end{displaymath}
	\caption{\label{fig:genconst2} An example of the construction in QYT$_{=8}(5,5,4,2)$.}
\end{figure}
Using the destandardization map, we can describe the relationship between quasi-Yamanouchi tableau and standard Young tableau. The following bijection can also be found in \cite{AS16} as a special case of Theorem $4.9$. 

\begin{proposition}\label{SYTbij}
Let $\lambda = (\lambda_1, \lambda_1, \ldots, \lambda_k)$ be a partition of $n$ and $m \geq n-(\lambda_1-1)$. Then
$$\textup{QYT}_{\leq m}(\lambda) \cong \textup{SYT}(\lambda)$$
\end{proposition}

\begin{proof}
When applying the destandardization map to a standard Young tableau $T$, it is easier to consider the runs of the tableau. The destandardization map decrements all elements in the first run of $T$ to $1$, all elements of the second run to $2$, all elements of the third run to $3$, etc. The result is clearly a quasi-Yamanouchi filling by the definition of runs; the start of the next run is always weakly left of the end of the previous.

Now we show that the map has an inverse. As we visit boxes, let the $k$th box that we visit have its entry replaced by $k$. The order of moving through boxes is as follows. Start at the leftmost instance of $1$, then move one by one to the right until all $1$s have been visited in order from left to right. Move to the leftmost instance of $2$, then move one by one to the right until all $2$s have been visited in order from left to right. Repeat this for $3, 4, \ldots$ until every box has been visited. The result must be a standard Young tableau. To show that it is the same one that we started with, observe that a standard filling is determined by the shape of the tableau and the way that the runs partition the filling. This inverse map clearly gives a standard Young tableau with runs exactly corresponding to the placement of $1$s,~$2$s,~$\ldots$ in the quasi-Yamanouchi filling, which in turn correspond to the runs of our original tableau $T$.
\end{proof}

The bijection in Proposition \ref{SYTbij} imposes some extra structure on standard Young tableaux. Since $\{\textrm{QYT}_{=m}(\lambda)\ |\ \ell(\lambda) \leq m \leq n-(\lambda-1)\}$ partitions the set of quasi-Yamanouchi fillings of $\lambda$, they must also partition the set of standard Young tableau of shape $\lambda$. In particular, they partition this set by separating the fillings by the number of runs they contain. 
The bijection also gives this useful result on the symmetry of quasi-Yamanouchi fillings. The following theorem tells us that if we can enumerate a given shape, we can also enumerate its conjugate.

\begin{theorem}\label{symmetry}
Given a partition $\lambda$ and its conjugate $\lambda'$ of size $n$, 

$$QYT_{=m}(\lambda) =QYT_{=(n+1)-m}(\lambda')$$
\end{theorem}

\begin{proof} Start with a quasi-Yamanouchi tableau $Q$ in QYT$_{=m}(\lambda)$ and restandardize it as in Proposition \ref{SYTbij}. The descent set of our standard Young tableau $T$ marks the ends of runs, except for the last run, so $|$Des$(T)| = m-1$. If we conjugate $T$ to get $T'$, then if $i+1$ was weakly left of $i$ in $T$, it becomes strongly right in $T'$, and if $i+1$ was strongly right of $i$ in $T$, it becomes weakly left in $T'$. Thus, the descent set of $T'$ is exactly the complement of Des$(T)\subseteq[n-1]$. The number of elements in Des$(T')$ then is $|$Des$(T')| = (n-1) - |$Des$(T)| = (n-1) - (m-1) = n-m$. This means that $T'$ has $n-m+1$ many runs. Then if we apply the destandardization map to $T'$, we will get a quasi-Yamanouchi tableau $Q'$ where the max value that appears is $(n+1)-m$. In other words, we get an element of QYT$_{=(n+1)-m}(\lambda')$. Since the restandardization map, conjugation map, and destandardization map are all bijective, we have a bjijection between QYT$_{=m}(\lambda)$ and QYT$_{=(n+1)-m}(\lambda')$. 
\end{proof}

\section{The main theorem}
Throughout this section, we will use $\lambda = (\lambda_1, \lambda_2, 2^{h_2-2}, 1^{h_1-h_2})$ to denote a partition with Durfee size $2$ and with first row length $\lambda_1$, second row length $\lambda_2$, first column height $h_1$, and second column height $h_2$. 
\begin{figure}[h]
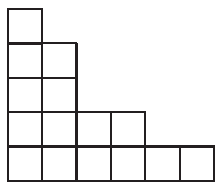

	\begin{displaymath}
	\tableau{
	\ \\
	\ & \ \\
	\ & \ \\
	\ & \ & \ & \ & \\
	\ & \ & \ & \ & \ & \ \\
	}
	\end{displaymath}
	\caption{\label{fig:notationex} The diagram corresponding to $(6, 4, 2^{4-2}, 1^{5-4}) = (6, 4, 2, 2, 1)$.}
\end{figure}

QYT$_{=m}(\lambda_1,\lambda_2, 2^{h_2-2}, 1^{h_1-h_2})$ will count the number of quasi-Yamanouchi tableau of Durfee size two with first row length $\lambda_1$, second row length $\lambda_2$, first column height $h_1$, second column height $h_2$, and maximum value that appears in the tableau exactly $m$.  The main goal of this paper is to prove the following theorem, which has a slightly nicer looking form in section 5.

\begin{theorem}\label{maintheorem}
Let $\lambda = (\lambda_1,\lambda_2,2^{h_2-2}, 1^{h_1-h_2})$ be a partition of size $n$ with $\lambda_1, \lambda_2,h_1,h_2 \geq 2$. If $h_1 \leq m \leq n-(\lambda_1 -1)$, then

\noindent\resizebox{.84\linewidth}{!}{
  \begin{minipage}{\linewidth}
\begin{equation*}
\begin{aligned}
&\textrm{QYT}_{=m}(\lambda) = \frac{\lambda_1-\lambda_2 +1}{m-h_2+1}{\lambda_1+h_1-2 \choose m-h_2}{\lambda_2+h_1-3 \choose m-h_2}{m-h_2\choose m-h_1}{\lambda_2 + h_2-4\choose h_2-2}{\lambda_1+h_2-3\choose h_2-2} \frac{1}{{h_1-1 \choose h_2-2}}
\end{aligned}
\end{equation*}
\end{minipage}
}
\end{theorem}

In the case where the tableau is of Durfee size $1$, it follows from Lemma \ref{SSYT} and Lemma \ref{range} that the set of quasi-Yamanouchi tableaux is equal to the set of semistandard Young tableaux where entries are bounded above by $h_1$.

Before we can approach the general case, there are some special cases which need to be addressed first. These cases involve diagrams where at least one of $\lambda_1, \lambda_2,h_1,$ or $h_2$ is exactly equal to $2$. 
When all four of $\lambda_1,\lambda_2,h_1,h_2 = 2$ so that $\lambda = (2, 2)$, it is easy to check by hand that the formula is correct. For cases which are conjugate shapes of other cases, only only one needs to be checked thanks to the following result, which essentially tells us that our formula implies Theorem \ref{symmetry}.

\begin{proposition}\label{symmetry2}
For $\lambda_1, \lambda_2, h_1, h_2, \geq 2$, $h_1 \leq m \leq h_1+h_2 + \lambda_2-3$, and $m' = \lambda_1+\lambda_2+h_1+h_2-3-m$, we have 

\noindent\resizebox{.95\linewidth}{!}{
  \begin{minipage}{\linewidth}
\begin{equation*}
\begin{aligned}
&\frac{\lambda_1-\lambda_2 +1}{m-h_2+1}{\lambda_1+h_1-2 \choose m-h_2}{\lambda_2+h_1-3 \choose m-h_2}{m-h_2\choose m-h_1}{\lambda_2 + h_2-4\choose h_2-2}{\lambda_1+h_2-3\choose h_2-2} \frac{1}{{h_1-1 \choose h_2-2}} =\\
&\frac{h_1-h_2+1}{m'-\lambda_2+1}{h_1+\lambda_1-2 \choose m' - \lambda_2}{h_2+\lambda_1-3\choose m'-\lambda_2}{m'-\lambda_2\choose m'-\lambda_1}{h_1+\lambda_2-3\choose \lambda_2-2}{h_2+\lambda_2-4\choose \lambda_2-2}\frac{1}{{\lambda_1-1\choose \lambda_2-2}}
\end{aligned}
\end{equation*}
\end{minipage}
}
\end{proposition}

\begin{proof} Equality comes easily after expanding the binomial coefficients into factorials and cancelling terms.
\end{proof}

We also have the following lemma, which proves that the theorem holds whenever $m$ is either minimal or maximal in the range given by Lemma \ref{range}. This lemma immediately proves the special cases where $\lambda_2 = h_2 = 2$ and also allows us to exclude the minimal and maximal values of the range from the statement of the general case. This is a computational convenience that allows that proof to go more smoothly.

\begin{lemma}\label{minmaxm}
Let $\lambda = (\lambda_1,\lambda_2, 2^{h_2-2}, 1^{h_1-h_2})$ be a partition of size $n$ with $\lambda_1,\lambda_2,$ $h_1,h_2 \geq 2$. Then if $m=h_1$ or $m = n-(\lambda_1-1)$, 

\noindent\resizebox{.84\linewidth}{!}{
  \begin{minipage}{\linewidth}
\begin{equation*}
\begin{aligned}
QYT_{=m}(\lambda) = \frac{\lambda_1-\lambda_2 +1}{m-h_2+1}{\lambda_1+h_1-2 \choose m-h_2}{\lambda_2+h_1-3 \choose m-h_2}{m-h_2\choose m-h_1}{\lambda_2 + h_2-4\choose h_2-2}{\lambda_1+h_2-3\choose h_2-2} \frac{1}{{h_1-1 \choose h_2-2}}
\end{aligned}
\end{equation*}
\end{minipage}
}
\end{lemma}

\begin{proof}
When $m=h_1$, the expression reduces algebraically to the hook-content formula for SSYT$_{h_1}(\lambda)$, and the result then follows from Lemma \ref{SSYT}. When $m=n-(\lambda_1-1)$, the expression must be equal to SSYT$_{\lambda_1}(\lambda')$ by Theorem \ref{symmetry}. Then applying Proposition \ref{symmetry2} and simplifying the right-hand side as before produces the hook-content formula for SSYT$_{\lambda_1}(\lambda')$.
\end{proof}

The remaining special cases have proofs which are essentially just easier versions of the general proof, so we collect all the special cases into one lemma.

\begin{lemma}\label{specialcases}
Let $\lambda = (\lambda_1, \lambda_2, 2^{h_2-2}, 1^{h_1-h_2})$ be a partition of size $n$ where at least one of $\lambda_1,\lambda_2,h_1,h_2 $ is equal to $2$. Then

\noindent\resizebox{.84\linewidth}{!}{
  \begin{minipage}{\linewidth}
$$QYT_{=m}(\lambda) = \frac{\lambda_1-\lambda_2 +1}{m-h_2+1}{\lambda_1+h_1-2 \choose m-h_2}{\lambda_2+h_1-3 \choose m-h_2}{m-h_2\choose m-h_1}{\lambda_2 + h_2-4\choose h_2-2}{\lambda_1+h_2-3\choose h_2-2} \frac{1}{{h_1-1 \choose h_2-2}}$$
\end{minipage}
}
\end{lemma}

We are now ready for the general case.

\textbf{Proof of Theorem \ref{maintheorem}.} Let $\lambda = (\lambda_1, \lambda_2, 2^{h_2-2}, 1^{h_1-h_2})$. If any of $\lambda_1, \lambda_2, h_1,$ or $h_2$ are equal to $2$, then we are done by Lemma \ref{specialcases}, and likewise if $m$ equals $h_1$ or $n-(\lambda_1-1)$, we can apply Lemma \ref{minmaxm}. Thus, we will consider when $\lambda_1, \lambda_2, h_1, h_2 \geq 3$ and $h_1 < m < n-(\lambda_1-1)$. We will proceed using induction, so assume that the theorem holds for partitions of size $< n$ and Durfee size $2$. There are four cases to consider, but we only need to check three. We have $h_1 > h_2$ and $\lambda_1 > \lambda_2$, $h_1 = h_2$ and $\lambda_1 > \lambda_2$, $h_1 > h_2$ and $\lambda_1=\lambda_2$, and finally $h_1 = h_2$ and $\lambda_1=\lambda_2$. We only need to check one of the second and third, since if the lemma holds for one, it holds for the other by Theorem \ref{symmetry} and Lemma \ref{symmetry2}.

Let $h_1 > h_2$ and $\lambda_1 > \lambda_2$. We have the four following possible ways to have $m$ in the first two columns.

\begin{figure}[h]
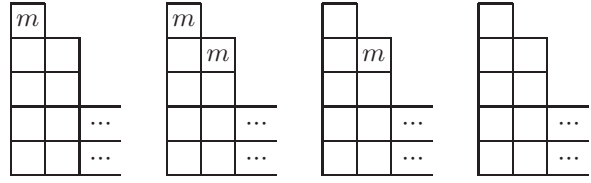

	\begin{displaymath}
	\tableau{
	m \\
	\ & \ \\
	\ & \ \\
	\ & \ & \mathclap{\whitebox}\mathclap{\raisebox{4\unitlength}{\ $\cdots$}}\\
	\ & \ & \mathclap{\whitebox}\mathclap{\raisebox{4\unitlength}{\ $\cdots$}}\\}
	\ \ \ \ \ \ 
	\tableau{
	m  \\
	\ & m \\
	\ & \ \\
	\ & \ & \mathclap{\whitebox}\mathclap{\raisebox{4\unitlength}{\ $\cdots$}}\\
	\ & \ & \mathclap{\whitebox}\mathclap{\raisebox{4\unitlength}{\ $\cdots$}}\\}
	\ \ \ \ \ \ 
	\tableau{
	\ \\
	\ & m \\
	\ & \ \\
	\ & \ & \mathclap{\whitebox}\mathclap{\raisebox{4\unitlength}{\ $\cdots$}}\\
	\ & \ & \mathclap{\whitebox}\mathclap{\raisebox{4\unitlength}{\ $\cdots$}}\\}
	\ \ \ \ \ \ 
	\tableau{
	\ \\
	\ & \ \\
	\ & \ \\
	\ & \ & \mathclap{\whitebox}\mathclap{\raisebox{4\unitlength}{\ $\cdots$}}\\
	\ & \ & \mathclap{\whitebox}\mathclap{\raisebox{4\unitlength}{\ $\cdots$}}\\}
	\end{displaymath}
	\caption{Possible arrangements of $m$.}
\end{figure}

Note that the last case is possible now that we can place an $m$ at the end of the second row, and in fact, this is forced by the quasi-Yamanouchi condition when $m$ is not in the first two columns. By summing over the ways that $m$ can appear in the first and second rows, we get the following recurrence. The function $A_{=m}(\lambda)$ counts the subset of quasi-Yamanouchi fillings of $\lambda$ where some $m$ appears anywhere but the first column. $B_{=m}(\lambda)$ counts the subset of quasi-Yamanouchi fillings where some $m$ appears anywhere but the first and second columns.

\noindent\resizebox{1\linewidth}{!}{
  \begin{minipage}{\linewidth}
\begin{equation*}
\begin{aligned}
\textrm{QYT}_{=m}(\lambda)= &\sum_{i=0}^{\lambda_2-2} \sum_{j=0}^{\lambda_1-\lambda_2} \textrm{QYT}_{=m-1}(\lambda_1-j, \lambda_2-i, 2^{h_2-2}, 1^{(h_1-1)-h_2})\\&+ \sum_{i=0}^{\lambda_2-2} \sum_{j=0}^{\lambda_1-\lambda_2} \textrm{QYT}_{=m-1}(\lambda_1-j, \lambda_2-i, 2^{(h_2-1)-2}, 1^{(h_1-1)-(h_2-1)}) \\&+\sum_{i=0}^{\lambda_2-2} \sum_{j=0}^{\lambda_1-\lambda_2} A_{=m-1}(\lambda_1-j, \lambda_2-i, 2^{(h_2-1)-2}, 1^{h_1-(h_2-1)})
\\&+\sum_{i=i}^{\lambda_2-2} \sum_{j=0}^{\lambda_1-\lambda_2} B_{=m-1}(\lambda_1-j, \lambda_2-i, 
2^{h_2-2}, 1^{h_1-h_2})
\end{aligned}
\end{equation*}
\end{minipage}
}

By definition of the $A_{=m}$ function,

\noindent\resizebox{.82\linewidth}{!}{
  \begin{minipage}{\linewidth}
\begin{equation*}
\begin{aligned}
A_{=m-1}(\lambda_1-j, \lambda_2-i,2^{(h_2-1)-2}, 1^{h_1-(h_2-1)}) = &\textrm{QYT}_{=m-1}(\lambda_1-j, \lambda_2-i, 2^{(h_2-1)-2}, 1^{h_1-(h_2-1)})\\& - \textrm{QYT}_{=m-2}(\lambda_1-j, \lambda_2-i, 2^{(h_2-1)-2}, 1^{(h_1-1)-(h_2-1)})
\end{aligned}
\end{equation*}
\end{minipage}
}

Comparing QYT$_{=m-1}(\lambda_1-j, \lambda_2-i, 2^{h_2-2}, 1^{h_1-h_2})$ with $B_{=m-1}(\lambda_1-j, \lambda_2-i,2^{h_2-2}, 1^{h_1-h_2})$ gives

\noindent\resizebox{.80\linewidth}{!}{
  \begin{minipage}{\linewidth}
\begin{equation*}
\begin{aligned}
\sum_{i=1}^{\lambda_2-2} \sum_{j=0}^{\lambda_1-\lambda_2} B_{=m-1}(\lambda_1-j, &\lambda_2-i, 2^{h_2-2}, 1^{h_1-h_2}) 
=\sum_{i=1}^{\lambda_2-2} \sum_{j=0}^{\lambda_1-\lambda_2} \textrm{QYT}_{=m-1}(\lambda_1-j,\lambda_2-i,2^{h_2-2}, 1^{h_1-h_2})
\\&-\sum_{i=1}^{\lambda_2-2} \sum_{j=0}^{\lambda_1-\lambda_2} \sum_{\ell = 0}^{\lambda_2-i-2} \textrm{QYT}_{=m-2}(\lambda_1-j, \lambda_2-i-\ell, 2^{h_2-2}, 1^{(h_1-1)-h_2})
\\&- \sum_{i=1}^{\lambda_2-2} \sum_{j=0}^{\lambda_1-\lambda_2} \sum_{\ell = 0}^{\lambda_2-i-2} \textrm{QYT}_{=m-2}(\lambda_1-j, \lambda_2-i-\ell, 2^{(h_2-1)-2}, 1^{(h_1-1)-(h_2-1)})
\\&- \sum_{i=1}^{\lambda_2-2} \sum_{j=0}^{\lambda_1-\lambda_2} \sum_{\ell = 0}^{\lambda_2-i-2} A_{=m-2}(\lambda_1-j, \lambda_2-i-\ell, 2^{(h_2-1)-2}, 1^{h_1-(h_2-1)})
\\&- \sum_{i=1}^{\lambda_2-2} \sum_{j=0}^{\lambda_1-\lambda_2} \sum_{\ell = 1}^{\lambda_2-i-2} B_{=m-2}(\lambda_1-j, \lambda_2-i-\ell, 2^{h_2-2}, 1^{h_1-h_2})
\\=\sum_{i=1}^{\lambda_2-2} \sum_{j=1}^{\lambda_1-\lambda_2} \textrm{QYT}_{=m-1}&(\lambda_1-j,\lambda_2-i,2^{h_2-2}, 1^{h_1-h_2})+\sum_{i=1}^{\lambda_2-2} \textrm{QYT}_{=m-1}(\lambda_2-1,\lambda_2-i,2^{h_2-2}, 1^{h_1-h_2})
\end{aligned}
\end{equation*}
\end{minipage}
}

The last step follows from comparison of the four triple sums with the expansion of 
$$\sum_{i=1}^{\lambda_2-2} [\textrm{QYT}_{=m-1}(\lambda_1,\lambda_2-i,2^{h_2-2}, 1^{h_1-h_2})-\textrm{QYT}_{=m-1}(\lambda_2-1,\lambda_2-i,2^{h_2-2}, 1^{h_1-h_2})]$$
followed by cancelling like terms. Substituting then gives

\resizebox{.85\linewidth}{!}{
  \begin{minipage}{\linewidth}
\begin{equation*}
\begin{aligned}
\textrm{QYT}_{=m}(\lambda)=& \sum_{i=0}^{\lambda_2-2} \sum_{j=0}^{\lambda_1-\lambda_2} \textrm{QYT}_{=m-1}(\lambda_1-j, \lambda_2-i, 2^{h_2-2}, 1^{(h_1-1)-h_2})
\\&+ \sum_{i=0}^{\lambda_2-2} \sum_{j=0}^{\lambda_1-\lambda_2} \textrm{QYT}_{=m-1}(\lambda_1-j, \lambda_2-i, 2^{(h_2-1)-2}, 1^{(h_1-1)-(h_2-1)})
\\&+\sum_{i=0}^{\lambda_2-2} \sum_{j=0}^{\lambda_1-\lambda_2}\textrm{QYT}_{=m-1}(\lambda_1-j, \lambda_2-i, 2^{(h_2-1)-2}, 1^{h_1-(h_2-1)})
\\&-\sum_{i=0}^{\lambda_2-2} \sum_{j=0}^{\lambda_1-\lambda_2}\textrm{QYT}_{=m-2}(\lambda_1-j, \lambda_2-i, 2^{(h_2-1)-2}, 1^{(h_1-1)-(h_2-1)})
\\&+\sum_{i=1}^{\lambda_2-2} \sum_{j=1}^{\lambda_1-\lambda_2} \textrm{QYT}_{=m-1}(\lambda_1-j,\lambda_2-i,2^{h_2-2}, 1^{h_1-h_2})
\\&+\sum_{i=1}^{\lambda_2-2} \textrm{QYT}_{=m-1}(\lambda_2-1,\lambda_2-i,2^{h_2-2}, 1^{h_1-h_2})
\end{aligned}
\end{equation*}
\end{minipage}
}

Finally, a comparison of $\textrm{QYT}_{=m}(\lambda_1,\lambda_2,2^{h_2-2}, 1^{h_1-h_2})-\textrm{QYT}_{=m}(\lambda_1-1,\lambda_2,2^{h_2-2}, 1^{h_1-h_2})$ and $\textrm{QYT}_{=m}(\lambda_1,\lambda_2-1,2^{h_2-2}, 1^{h_1-h_2})-\textrm{QYT}_{=m}(\lambda_1-1,\lambda_2-1,2^{h_2-2}, 1^{h_1-h_2})$ gives

\noindent\resizebox{.85\linewidth}{!}{
  \begin{minipage}{\linewidth}
\begin{equation*}
\begin{aligned}
\textrm{QYT}&_{=m}(\lambda_1,\lambda_2,2^{h_2-2}, 1^{h_1-h_2})\\& = \textrm{QYT}_{=m}(\lambda_1-1,\lambda_2,2^{h_2-2}, 1^{h_1-h_2}) + \textrm{QYT}_{=m}(\lambda_1,\lambda_2-1,2^{h_2-2}, 1^{h_1-h_2})
\\&- \textrm{QYT}_{=m}(\lambda_1-1,\lambda_2-1,2^{h_2-2}, 1^{h_1-h_2}) + \textrm{QYT}_{=m-1}(\lambda_1,\lambda_2,2^{h_2-2}, 1^{(h_1-1)-h_2})
\\&+\textrm{QYT}_{=m-1}(\lambda_1,\lambda_2,2^{(h_2-1)-2}, 1^{(h_1-1)-(h_2-1)}) + \textrm{QYT}_{=m-1}(\lambda_1,\lambda_2,2^{(h_2-1)-2}, 1^{h_1-(h_2-1)})
\\&-\textrm{QYT}_{=m-2}(\lambda_1,\lambda_2,2^{(h_2-1)-2}, 1^{(h_1-1)-(h_2-1)}) + \textrm{QYT}_{=m-1}(\lambda_1-1,\lambda_2-1,2^{h_2-2}, 1^{h_1-h_2})
\end{aligned}
\end{equation*}
\end{minipage}
}

Replace each term with the right-hand expression in the lemma statement. Expanding this into factorials, multiplying out denominators, and cancelling terms gives equality.

When $h_1>h_2$ and $\lambda_1 = \lambda_2$, we consider the same diagrams as in the first case. We repeat the same argument step by step, while taking care to note which terms will be zero. For example, there are no quasi-Yamanouchi tableau corresponding to the shape $(\lambda_1-1, \lambda_1,2^{h_2-2}, 1^{h_1-h_2})$, so $\textrm{QYT}_{=m}(\lambda_1-1, \lambda_1, 2^{h_2-2}, 1^{h_1-h_2}) = 0$. There are also now empty sums in some of the recurrence expansions to watch out for. In the end, we will get a similar equation as before.

\noindent\resizebox{.85\linewidth}{!}{
  \begin{minipage}{\linewidth}
\begin{equation*}
\begin{aligned}
\textrm{QYT}&_{=m}(\lambda_1,\lambda_1,2^{h_2-2}, 1^{h_1-h_2})
\\& = \textrm{QYT}_{=m}(\lambda_1,\lambda_1-1,2^{h_2-2}, 1^{h_1-h_2}) - \textrm{QYT}_{=m}(\lambda_1-1,\lambda_1-1,2^{h_2-2}, 1^{h_1-h_2}) 
\\&+ \textrm{QYT}_{=m-1}(\lambda_1,\lambda_1,2^{h_2-2}, 1^{(h_1-1)-h_2})+\textrm{QYT}_{=m-1}(\lambda_1,\lambda_1,2^{(h_2-1)-2}, 1^{(h_1-1)-(h_2-1)})
\\& + \textrm{QYT}_{=m-1}(\lambda_1,\lambda_1,2^{(h_2-1)-2}, 1^{h_1-(h_2-1)})-\textrm{QYT}_{=m-2}(\lambda_1,\lambda_1,2^{(h_2-1)-2}, 1^{(h_1-1)-(h_2-1)})
\\& + \textrm{QYT}_{=m-1}(\lambda_1-1,\lambda_1-1,2^{h_2-2}, 1^{h_1-h_2})
\end{aligned}
\end{equation*}
\end{minipage}
}

Once again, write each term out as a product of factorials and multiply out denominators. Expanding products and then canceling shows that these are equal.

Finally, let $h_1 = h_2$ and $\lambda_1 = \lambda_2$. Now there are only three diagrams that will contribute non-zero terms. 

\begin{figure}[h]
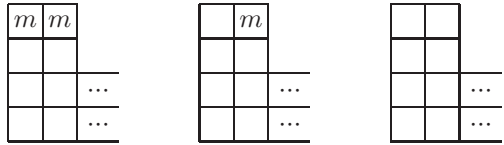

	\begin{displaymath}
	\tableau{
	m&m\\
	\ & \ \\
	\ & \ & \mathclap{\whitebox}\mathclap{\raisebox{4\unitlength}{\ $\cdots$}}\\
	\ & \ & \mathclap{\whitebox}\mathclap{\raisebox{4\unitlength}{\ $\cdots$}}\\}
	\ \ \ \ \ \ \ \ \ \ \tableau{
	\ &m\\
	\ & \ \\
	\ & \ & \mathclap{\whitebox}\mathclap{\raisebox{4\unitlength}{\ $\cdots$}}\\
	\ & \ & \mathclap{\whitebox}\mathclap{\raisebox{4\unitlength}{\ $\cdots$}}\\}
	\ \ \ \ \ \ \ \ \ \ \tableau{
	\ &\ \\
	\ & \ \\
	\ & \ & \mathclap{\whitebox}\mathclap{\raisebox{4\unitlength}{\ $\cdots$}}\\
	\ & \ & \mathclap{\whitebox}\mathclap{\raisebox{4\unitlength}{\ $\cdots$}}\\}	
	\end{displaymath}
	\caption{Possible arrangements of $m$.}
\end{figure}

We use a similar argument again, where now more terms vanish. After the dust settles, we find that

\noindent\resizebox{.85\linewidth}{!}{
  \begin{minipage}{\linewidth}
\begin{equation*}
\begin{aligned}
\textrm{QYT}&_{=m}(\lambda_1,\lambda_1,2^{h_1-2}, 1^{h_1-h_1}) \\&= \textrm{QYT}_{=m}(\lambda_1,\lambda_1-1,2^{h_1-2}, 1^{h_1-h_1})-\textrm{QYT}_{=m}(\lambda_1-1,\lambda_1-1,2^{h_1-2}, 1^{h_1-h_1})
\\&+\textrm{QYT}_{=m-1}(\lambda_1,\lambda_1,2^{(h_1-1)-2}, 1^{(h_1-1)-(h_1-1)}) + \textrm{QYT}_{=m-1}(\lambda_1,\lambda_1,2^{(h_1-1)-2}, 1^{h_1-(h_1-1)})
\\&-\textrm{QYT}_{=m-2}(\lambda_1,\lambda_1,2^{(h_1-1)-2}, 1^{(h_1-1)-(h_1-1)}) +\textrm{QYT}_{=m-1}(\lambda_1-1,\lambda_1-1,2^{h_1-2}, 1^{h_1-h_1})
\end{aligned}
\end{equation*}
\end{minipage}
}

Expand and cancel one last time and we are done. \hfill \qed

\section{Larger shapes and future work}

The work in this paper is for a purely algebraic proof, but we can also write $\textrm{QYT}_{=m}(\lambda)$ in a manner that suggests possible bijective algorithms. For example, with a little manipulation we have

\noindent\resizebox{.95\linewidth}{!}{
  \begin{minipage}{\linewidth}
\begin{equation*}
\begin{aligned}
{m-h_2+1\choose h_1-h_2+1}\textrm{QYT}_{=m}&(\lambda_1,\lambda_2,2^{h_2-2}, 1^{h_1-h_2}) 
\\&=
{\lambda_1+h_2-2\choose m-h_1}{\lambda_2+h_2-3\choose m-h_1}\text{\normalfont SSYT}_{h_1}(\lambda_1,\lambda_2,2^{h_2-2}, 1^{h_1-h_2})
\end{aligned}
\end{equation*}
\end{minipage}
}

For those who are a little more brave, we also have 

\noindent\resizebox{.94\linewidth}{!}{
  \begin{minipage}{\linewidth}
\begin{equation*}
\begin{aligned}
{m+\lambda_1-1\choose m-h_1}{m+\lambda_2-2\choose m-h_1}&\textrm{QYT}_{=m}(\lambda_1,\lambda_2,2^{h_2-2}, 1^{h_1-h_2}) =
\\&{\lambda_1+h_2-2\choose m-h_1}{\lambda_2+h_2-3\choose m-h_1}\text{\normalfont SSYT}_{m}(\lambda_1,\lambda_2,2^{h_2-2}, 1^{h_1-h_2})
\end{aligned}
\end{equation*}
\end{minipage}
}

Of course, the general case is also open. We provide a few examples thanks to computer calculations.
\begin{figure}[H]
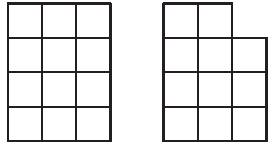

	\begin{displaymath}
	\tableau{
	\ & \ & \ \\
	\ & \ & \ \\
	\ & \ & \ \\
	\ & \ & \ \\}
	\ \ \ \ \ \ 
	\tableau{
	\ & \ \\
	\ & \ & \ \\
	\ & \ & \ \\
	\ & \ & \ \\}
	\end{displaymath}
	\caption{On the left, $\textrm{QYT}_{=6}(\lambda) = 113$ and on the right, QYT$_{\leq6}(\lambda) = 241$.}
\end{figure}
In these fairly small examples, we are faced with relatively large primes. The first example serves to demonstrate that a product formula is unlikely to exist when attempting to enumerate quasi-Yamanouchi fillings where $m$ is the largest value that appears, while the second demonstrates the same for attempts to enumerate quasi-Yamanouchi fillings that are bounded above by $m$. Thus, the general case will most likely require some form of summation. We end with the following table that counts quasi-Yamanouchi fillings of different partitions of $6$, which visually demonstrates their symmetry.
\begin{figure}[H]
\begin{tabular}{r|>{\centering}p{8\unitlength}|>{\centering}p{8\unitlength}|>{\centering}p{8\unitlength}|>{\centering}p{8\unitlength}|>{\centering}p{8\unitlength}|c|}
\multicolumn{1}{c}{}&\multicolumn{1}{c}{}&\multicolumn{1}{c}{}&\multicolumn{2}{c}{\underline{m}}\\
\multicolumn{1}{c}{}& \multicolumn{1}{c}{1}&\multicolumn{1}{c}{2}&\multicolumn{1}{c}{3}&\multicolumn{1}{c}{4}&\multicolumn{1}{c}{5}&\multicolumn{1}{c}{6}\\ \cline{2-7}
$(6)$ &$1$ & & & & & \\ \cline{2-7}
$(5, 1)$ & &$5$ & & & & \\ \cline{2-7}
$(4, 2)$ & &$3$ &$6$ & & & \\ \cline{2-7}
$(4, 1, 1)$ & & & $10$& & & \\ \cline{2-7}
$(3, 3)$ & & $1$& $3$& $1$& & \\ \cline{2-7}
\rotatebox{90}{\rlap{\underline{$\lambda \vdash 6$}}}\ \ \ \ \  $(3, 2, 1)$ & & & $8$&$8$ & & \\ \cline{2-7}
$(2, 2, 2)$ & & & $1$&$3$ &$1$ & \\ \cline{2-7}
$(3, 1, 1, 1)$ & & & & $10$& & \\ \cline{2-7}
$(2, 2, 1, 1)$ & & & & $6$&$3$ & \\ \cline{2-7}
$(2, 1, 1, 1, 1)$ & & & & & $5$& \\ \cline{2-7}
$(1, 1, 1, 1, 1, 1)$ & & & & & & $\hspace{1.5\unitlength}1\hspace{1.5\unitlength}$\\ \cline{2-7}
\end{tabular}
\vspace{2.5mm}
\caption{A table that gives the values of $\textrm{QYT}_{=m}(\lambda)$ for each partition of $6$ and each value of $1\leq m \leq 6$. Blank entries indicate a $0$ value.}
\end{figure}

\section*{Acknowledgements}

I am grateful to Sami Assaf for her valuable guidance and encouragement. I would also like to thank Dominic Searles for his helpful advice. Computations were done in C++.

\providecommand{\bysame}{\leavevmode\hbox to3em{\hrulefill}\thinspace}
\providecommand{\MR}{\relax\ifhmode\unskip\space\fi MR }
\providecommand{\MRhref}[2]{%
  \href{http://www.ams.org/mathscinet-getitem?mr=#1}{#2}
}
\providecommand{\href}[2]{#2}


\begin{thebibliography}{EML53}

\bibitem[AS16]{AS16}
Sami Assaf and Dominic Searles, \emph{Schubert polynomials, slide polynomials, Stanley symmetric functions and quasi-Yamanouchi pipe dreams}, {\tt arXiv:1603.09744 [math.CO]} (2016).


\bibitem[Ge84]{Ge84}
Ira~M. Gessel, \emph{Multipartite {$P$}-partitions and inner products of skew
  {S}chur functions}, Combinatorics and algebra (Boulder, Colo., 1983),
  Contemp. Math., vol.~34, Amer. Math. Soc., Providence, RI, 1984,
  pp.~289--317.

\bibitem[FRT54]{FRT54}
Frame, J. S. and Robinson, G. de B. and Thrall, R. M., \emph{The hook graphs of the symmetric groups}, Canadian J. Math., vol.~6, pp.~316--324. \MR{0062127 (15,931g)}

\bibitem[Sta71]{Sta71}
Richard~P. Stanley, \emph{Theory and application of plane partitions. {I}, {II}}, Studies in Appl. Math., vol.~50, pp.~167--188; ibid. 50 (1971), 259--279. \MR{0325407 (48 \#3754)}

\end{thebibliography}
\end{document}